 \documentclass[12pt,{other stuff},reqno,oneside]{amsart}

\numberwithin{equation}{subsection}

   \usepackage{amsmath}       
 \usepackage{amssymb} 
\theoremstyle{plain}  
      \newtheorem{theorem}{Theorem}[section]
      \newtheorem{lemma}[theorem]{Lemma}
      \newtheorem{corollary}[theorem]{Corollary}
      \newtheorem{proposition}[theorem]{Proposition}
      \theoremstyle{definition}

      \usepackage[usenames,dvipsnames]{color}
   \usepackage[pdftex]{graphicx}
  
 \usepackage{colortbl} 
      
\begin{document}
\author{N. Ghroda}
  \address{Department of Mathematics\\University
  of York\\Heslington\\York YO10 5DD\\UK}
\email{ng521@york.ac.uk}

    \title[Bisimple inverse $\omega$-semigroups of left I-quotients]{Bisimple inverse  $\omega$-semigroups of left I-quotients}

   \begin{abstract}A subsemigroup  $S$  of an inverse semigroup  $Q$  is a  left I-order in  $Q$ if every element in  $Q$  can be written as  $a^{-1}b$   where  $a ,b \in S$  and  $a^{-1}$  is the inverse of  $a$  in the sense of inverse semigroup theory. If we insist on  $a$ and  $b$  being  $\mathcal{R}$-related in  $Q$, then we say that  $S$ is a straight left I-order in  $Q$.  We give necessary and sufficient conditions for a semigroup to be a left I-order in a bisimple inverse  $\omega$-semigroup.
   \end{abstract}
   

   %\thanks{All prais are due to Allah }
   \date{\today}

   \maketitle
 
\section{Introduction}\label{Sec1}
Many definitions of semigroups of quotients have been proposed and studied. The first, that was specifically tailored to the structure of semigroups was introduced  by Fountain and Petrich in \cite{pjhon}, but was restricted to completely 0-simple semigroups of left quotients. This definition has been extended to the class of all semigroups \cite{bisGould}. The idea is that a subsemigroup  $S$  of a semigroup $Q$  is a \emph{left order} in  $Q$   or  $Q$  is a \emph{semigroup of left quotients} of $S$ if every element of  $Q$  can be written as   $a^{\sharp}b$  where  $a , b \in S$  and  $a^{\sharp}$  is the inverse of   $a$  in a subgroup of  $Q$ and if, in addition, every \emph{square-cancellable} element  (an element  $a$ of a semigroup  $S$ is square-cancellable if  $a\, \mathcal{H}^{*}\,a^{2}$) lies in a subgroup of  $Q$. \emph{Semigroups of right quotients} and \emph{right orders} are defined dually. If  $S$ is both a left order and a right order in a semigroup  $Q$, then  $S$  is an \emph{order} in  $Q$ and $Q$ is a semigroup of \emph{quotients} of $S$. This definition and its dual were used in \cite{bisGould} to characterize semigroups which have  bisimple inverse $\omega$-semigroups of left quotients.\par

\bigskip 
On the other hand, Clifford \cite{clifford} showed that from any right cancellative monoid  $S$ with (LC) there is a bisimple inverse  monoid $Q$ such that  $Q=S^{-1}S$; that is, every element  $q$  in $Q$ can be written as  $a^{-1}b$   where  $a ,b \in S$ and $a^{-1}$ is the inverse of $a$ in $Q$ in the sense of inverse semigroup theory. By saying that a semigroup  $S$ has the (LC) \emph{condition} we mean that for any  $a,b\in S$ there is an element  $c\in S$ such that  $Sa\cap Sb=Sc$.  The author and Gould in \cite{GG} have  extended Clifford's work to a left ample semigroup with (LC) where they  introduced the following definition of left I-orders in inverse semigroups:\par
 \bigskip
 Let  $Q$ be an inverse semigroup. A  subsemigroup  $S$  of  $Q$  is a \emph{left I-order} in  $Q$ or $Q$ is a semigroup of \emph{left I-quotients} of $S$, if every element in  $Q$  can be written as  $a^{-1}b$    where  $a ,b \in S$. The notions of \emph{right I-order} and  \emph{semigroup of right I-quotients} are defined dually. If  $S$ is both a left I-order and a right I-order in  $Q$, we say that  $S$  is an \emph{I-order} in  $Q$ and $Q$ is a semigroup of \emph{I-quotients} of $S$. It is clear that, if $S$ a left order in an inverse semigroup $Q$, then it is certainly a left I-order in $Q$; however, the converse is not true   (see for example \cite{GG} Example 2.2). \par
\bigskip 
A  left I-order $S$ in an inverse semigroup $Q$ is  \emph{straight left I-order}if every element in  $Q$ can be written as  $a^{-1}b$ where  $a,b \in S$ and  $a\,\mathcal{R}\,b$ in  $Q$;  we also say that $Q$ is a \emph{ semigroup of straight left I-quotients} of $S$. If  $S$ is  straight in $Q$,
we have the advantage of controlling  products in $Q$. Many left I-orders are straight, such as left I-orders in  primitive inverse semigroups. In the case where $S$ is a straight left I-order in $Q$, the uniqueness problem has been solved \cite{GG}, that is, the author and Gould have given  necessary and sufficient conditions for a left I-order to have a unique semigroup of a  left I-quotients. \par
\bigskip 
In \cite{onordersGould} it was shown that if $\mathcal{H}$ is a congruence on a regular semigroup $Q$, then every left order $S$ in $Q$ is
straight. To prove this, Gould uses the fact that  $S$  intersects every $\mathcal{H}$-class of $Q$. Since $\mathcal{H}$ is congruence on any
bisimple inverse $\omega$-semigroup,  any left order $S$ in  such a semigroup must be straight. In the case of left I-orders we show that if $S$ is a left I-order in a bisimple inverse $\omega$-semigroup $Q$, then $S$ intersects every $\mathcal{L}$-class of $Q$ and we use this to show that $S$ is straight in $Q$.\par

\bigskip

 The main aim of this article is to study semigroups which have  bisimple inverse $\omega$-semigroups of  left I-quotients. After giving some preliminaries in Section~\ref{Sec2}, in Section~\ref{Sec3} we extend the result in \cite{bisGould}, by introducing the main theorem in this article, which gives necessary and sufficient conditions for a semigroup to be a left I-order in a bisimple inverse  $\omega$-semigroup.

\newpage

\section{Preliminaries}\label{Sec2}
   
  Throughout this article we shall follow the terminology and notation of \cite{clifford}. The set of all non-negative integers will be denoted by $\mathbb{N}^0$. 
 Let $ \mathcal{R} ,  \mathcal{L} ,  \mathcal{H}$ and  $\mathcal{D}=  \mathcal{R} \circ \mathcal{L}=\mathcal{L} \circ \mathcal{R}$\ be the usual Green's relations. A semigroup  $S$ is\emph{ bisimple} if it consists of a single  $\mathcal{D}$-class.\par
  \bigskip 
  For any semigroup  $Q$ with the set of idempotents  $E$ we define a partial ordering  $\leqslant$ on  $E$ by the rule that  $e\leqslant f$ if and only if  $ef=fe=e$. A \emph{bisimple inverse  $\omega$-semigroup} is a bisimple inverse  semigroup whose idempotents form
an  $\omega$-chain; that is,  $E(S)=\{e_{m}: m \in \mathbb{N}^0\}$ where $ e_{0} \geqslant e_{1} \geqslant e_{2} \geqslant \ldots $. Thus if  $S$ is a bisimple inverse  $\omega$-semigroup, on  $E$ we have \[e_{m} \leqslant e_{n}\ \mbox{if and only if} \ m\geqslant n.\]
Reilly \cite{Reilly} determined the structure of all bisimple inverse  $\omega$-semigroups as follows: \\
 
Let  $G$ be a group and let  $\theta$ be an endomorphism of  $G$. Let  $S(G,\theta)$ be the semigroup on $\mathbb{N}^0 \times G\times \mathbb{N}^0$ with multiplication \[(m,g,n)(p,h,q)=\big(m-n+t,(g\theta^{t-n})(h\theta^{t-p}),q-p+t\big)\] where  $t$=max$\{n,p\}$ and  $\theta^0$ is interpreted as the identity map of  $G$. As was shown in \cite{Reilly} (Cf \cite{howie}),  $S(G,\theta)$ is a bisimple inverse  $\omega$-semigroup and every bisimple inverse $\omega$-semigroup is isomomorphic to  $S(G,\theta)$. In the case where  $G$ is trivial, then  $S(G,\theta)=\mathcal{B}$ where  $\mathcal{B}$  the \emph{bicyclic semigroup}, we identify $\mathcal{B}$ with $\mathbb{N}^0 \times \mathbb{N}^0$. The set of idempotents of  $S(G,\theta)$ is  $\{(m,1,m); m \in \mathbb{N}^0\}$ and for any  $(m,g,n)$ in $S(G,\theta)$, \[(m,g,n)^{-1}=(n,g^{-1},m).\]
  For any  $(m,a,n), (p,b,q)\in S(G,\theta)$, 
   \[(m,a,n)\,\mathcal{R}\,(p,b,q)\ \mbox{if and only if} \ m=p,\]
   \[(m,a,n)\,\mathcal{L}\,(p,b,q)\ \mbox{if and only if} \ n=q,\]
   and, consequently, 
   \[(m,a,n)\,\mathcal{H}\,(p,b,q)\ \mbox{if and only if} \ m=p\ \mbox{and}\ n=q.\]
   If  $Q$ is a bisimple inverse $\omega$-semigroup, let  $R_n$ ($L_n$) denote the  $\mathcal{R}$-class ($\mathcal{L}$-class) of  $Q$ containing the idempotent  $e_n=(n,1,n)$. From the above,
   \[R_m=\{(m,a,n): a\in G, n \in \mathbb{N}^0\}, \]
    \[L_n=\{(m,a,n): a\in G, m \in \mathbb{N}^0\}. \]
     Clearly,
     \[\begin{array}{rcl}H_{m,n}=R_m \cap L_n &=& \{(m,a,n): a\in G\}\\ &=&\{ q\in Q:  qq^{-1}=e_m , q^{-1}q=e_n\} \end{array} \]
   and from the multiplication in $S(G,\theta)$, 
     \[H_{m,n}H_{p,q}\subseteq H_{m-n+t,q-p+t},\]
    where  $t=$max$\{n,p\}$.\par
    \bigskip
Let $S$ be any semigroup such that there is a homomorphism $\varphi: S\longrightarrow \mathcal{B}$. We define functions $l,r:S \longrightarrow \mathbb{N}^0$ by  $a\varphi=\big(r(a),l(a)\big)$. We also put $H_{i,j}=(i,j)\varphi^{-1}$, so that $S$ is a disjoint union of subsets of the $H_{i,j}$ and \[H_{i,j}=\{a\in S: r(a)=i,l(a)=j\}.\]

From the above,  $\mathcal{H}$ is a congruence on any bisimple inverse  $\omega$-semigroup  $Q$. Moreover there is a homomorphism  $\overline{\varphi} :Q\longrightarrow \mathcal{B}$  given  by $(m,p,n)\overline{\varphi}=(m,n)$ which is  surjective  with  $Ker \overline{\varphi}=\mathcal{H}$ so  $Q/\mathcal{H}\cong \mathcal{B}$ where  $\mathcal{B}$ is the bicyclic semigroup. As above  we will index   $\mathcal{H}$ in  $Q$ by putting  $H_{i,j}=(i,j)\overline{\varphi}^{-1}$.\par 
\bigskip
Let  $S$ be a left I-order in  $Q$. Let $\varphi=\overline{\varphi}|_S$, then  $\varphi$ is a homomorphism from  $S$ to  $\mathcal{B}$. Unfortunately, this homomorphism is not surjective in general, since  $S$ need not intersect every  $\mathcal{H}$-class of  $Q$. But we can as above index the elements of  $S$.\par
\bigskip
In \cite{NGb} it was shown that, if a semigroup $S$ is a left I-order in a bicyclic semigroup $\mathcal{B}$, then $S$ interesects every $\mathcal{L}$-class of $\mathcal{B}$. Moreover, it is straight. In fact, this is true for any left I-order in a bisimple inverse $\omega$-semigroup, as we will see in the next lemmas.

\begin{lemma}\label{lintersect}
If a semigroup $S$ is a left I-order in a bisimple inverse $\omega$-semogroup $Q$, then $S\cap L_n\neq\emptyset$ for all $n \in \mathbb{N}^0$.\end{lemma}
\begin{proof}
Let $p\in H_{n,n}$, then $p=a^{-1}b$ for some $a,b\in S$ with $a\in H_{i,j}$ and $b\in H_{k,l}$. Hence \[p=a^{-1}b\in H_{j,i}H_{k,l}\subseteq H_{j-i+max(i,k),l-k+max(i,k)},\] and so $n=j-i+max(i,k)=l-k+max(i,k)$. So, as max$(i,k)=i$ or $k$, either $n=j$ or $n=l$. Hence $S\cap L_n\neq \emptyset$.\end{proof}

In \cite{bisGould} it was shown that if $S$ a left order in a bisimple inverse $\omega$-semigroup $Q$, then it is straight. The following lemma extends this to the left I-order in bisimple inverse $\omega$-semigroup.
\begin{lemma}\label{straightlaw}If a semigroup $S$ is a left I-order in a bisimple inverse $\omega$-semigroup $Q$, then $S$ is straight. \end{lemma}
\begin{proof}
Let $(h,q,k) \in Q$, then $(h,q,k)=(i,a,j)^{-1}(t,b,s)=(j,a^{-1},i)(t,b,s)$ for some $(i,a,j),(t,b,s)\in S$. Let $n=$max$\{i,t\}$; since $S\cap L_n\neq \emptyset$, by Lemma ~\ref{lintersect}, there exist $(u,c,n)\in S\cap L_n$ and hence $(u,c,n)^{-1}(u,c,n)=(n,1,n)$, so that $(n,1,n)\,\mathcal{R}\,(t,b,s)$ or $(n,1,n)\,\mathcal{R}\,(i,a,j)$. In both cases, we have \[\begin{array}{rcl}(h,q,k)&=&(i,a,j)^{-1}(n,1,n)(t,b,s)\\ &=&(i,a,j)^{-1}(u,c,n)^{-1}(u,c,n)(t,b,s)\\ &=&\big((u,c,n)(i,a,j)\big)^{-1}\big((u,c,n)(t,b,s)\big). \end{array}\]It is clear that 
$(u,c,n)(i,a,j)\,\mathcal{R}\,(u,c,n)(t,b,s)$. Hence $S$ is straight.\end{proof}
 
\begin{proposition}\label{straightprop}Let $Q$ be an inverse semigroup and $q=a^{-1}b$  with $a\,\mathcal{R}\,b$, then $a^{-1}\,\mathcal{R}\,q\,\mathcal{L}\,b$. \end{proposition}
The following corollaries are clear.
\begin{corollary}\label{related}
Let $Q$ be an inverse semigroup. If $a^{-1}b,c^{-1}d \in Q$ where $a\,\mathcal{R}\,b$ and  $c\,\mathcal{R}\,d$, then\\
$(i)$\ $a^{-1}b\,\mathcal{R}\,c^{-1}d  \Longleftrightarrow  a^{-1}a=c^{-1}c$;\\
$(ii)$\ $a^{-1}b\,\mathcal{L}\,c^{-1}d  \Longleftrightarrow  b^{-1}b=d^{-1}d$. \end{corollary}
\begin{corollary}\label{biswrelated}
Let  $Q$ be a bisimple inverse  $\omega$-semigroup, then \\
$(i)$\ $(m,a,n)^{-1}(m,b,t)\,\mathcal{R}\,(i,c,j)^{-1}(i,d,k)$  if and only if $n=j$;\\
$(ii)$\ $(m,a,n)^{-1}(m,b,t)\,\mathcal{L}\,(i,c,j)^{-1}(i,d,k)$  if and only if $t=k$.\end{corollary}
 
\section{The main theorem}
  \label{Sec3}
This section is entirely devoted to proving Theorem ~\ref{bisimplew} which gives a characterisation of semigroups which have a bisimple inverse $\omega$-semigroup of left I-quotients.

\begin{theorem}\label{bisimplew}
A semigroup  $S$  is a left I-order in a bisimple inverse w-semigroup  $Q$  if and only if  $S$  satisfies the following conditions:\\
$(A)$ There is a homomorphism  $\varphi:S\longrightarrow \mathcal{B}$ such that $S\varphi$ is a left I-order in $\mathcal{B}$;\\
$(B)$ For  $x   , y , a \in S$,  
\[(i) \ l(x), l(y) \geqslant r(a)\ \mbox{and} \ xa=ya \ \mbox{implies} \ x=y, \]
\[(ii) \ r(x), r(y) \geqslant l(a)\ \mbox{and} \ ax=ay \ \mbox{implies} \ x=y. \]
$(C)$ For any $b,c\in S$, there exist \ $x, y \in S$\ such that $xb = yc$ where
\[x \in H_{r(x),r(b)-l(b)+max\big(l(b), l(c)\big)},  y \in H_{r(x),r(c)-l(c)+max\big(l(b), l(c)\big)}.\]
\end{theorem}
\begin{proof}
Let $S$ be a left I-order in a bisimple inverse $\omega$-semigroup $Q$. For condition $(A)$, since $S$ is a left I-order in $Q$ and there is a homomorphism  
$\overline{\varphi}: Q \rightarrow {\mathcal{B}}$ given by
\[(m, p, n)\overline{\varphi} = (m,n),\]
we can restrict $\overline{\varphi}$ on $S$ to get a homomorphism $\varphi$ from $S$ to $\mathcal{B}$. Let
$(i, j) \in {\mathcal{B}}$, then there is an element $q$ in $Q$ such that $q \in H_{i, j}$ for some $i, j \in \mathbb{N}^0$. Put $q = a^{-1}b$ for some $a, b \in S$ with $a\,\mathcal{R}\,b$ in $Q$, so that $r(a)=r(b)$. Hence
 \[q \in H_{l(a),r(a)}H_{r(a),l(b)} \subseteq H_{l(a),l(b)},\]
 then \[\begin{array}{rcl}(i, j) &= &\big(l(a),l(b)\big)\\ & =& \big(r(a),l(a)\big)^{-1}\big(r(b),l(b)\big)\\ & =& (a\varphi)^{-1}(b\phi).\end{array}\]
 
 To see that $(B)(i)$ holds, suppose that $x ,y ,a \in S$ where $l(x),l(y) \geqslant r(a)$ and $xa = ya$. Since $a^{-1} \in H_{l(a),r(a)}$ and $xaa^{-1} = yaa^{-1}$, that is, $xe_{r(a)} = ye_{r(a)}$, and $r(a) \leqslant l(x), l(y)$, then we have $e_{l(x)}, e_{l(y)} \leqslant e_{r(a)}$. Hence $xe_{l(x)}e_{r(a)} = ye_{l(y)}e_{r(a)}$ and so  $x = xe_{l(x)} = ye_{l(y)} = y$. \par
 \medskip
 $(B)(ii)$ Similar to $(B)(i)$.\\ 
\\
Finally, we consider $(C)$. Let $b, c \in S$, then $bc^{-1} \in Q$ and
 \[bc^{-1} \in H_{r(b),l(b)}H_{l(c),r(c)} \subseteq H_{r(b)-l(b)+max\big(l(b),l(c)\big),r(c)-l(c)+max\big(l(b),l(c)\big)}.\]
 Since $S$ is a straight left I-order in $Q$, then $bc^{-1} = x^{-1}y$ where $x\,\mathcal{R}\,y$ for some $x, y \in S$, and by Lemma 2.6 in \cite{GG}, $xb = yc$. From $bc^{-1}=x^{-1}y$ we have 
  \[H_{r(b)-l(b)+max\big(l(b),l(c)\big),r(c)-l(c)+max\big(l(b),l(c)\big)} = H_{l(x), l(y)},\]
 so that $l(x) = r(b)-l(b)+max\big(l(b),l(c)\big)$ and $l(y) = r(c)-l(c)+max\big(l(b),l(c)\big)$. \\
 \\
  
  Conversely, we suppose that $S$ satisfies conditions $(A) , (B)$ and $(C)$. Now, our aim is to construct via equivalence classes of order pairs of elements of  $S$ an inverse semigroup $Q$, which is a semigroup of straight left I-quotients of  $S$. First, we let 
\[\Sigma= \{(a, b)\in S \times S: r(a) = r(b)\}\]
and on $\Sigma$ we define the relation $\sim$  as follows: 
  \[(a, b) \sim (c,d) \Leftrightarrow \ \mbox{there are elements} \ x, y\ \mbox{in}\ S\ \mbox{such that}\ xa =yc \ \mbox{and} \ xb= yd\]
  where $l(x)=r(a),l(y)= r(c)$ and $r(x) = r(y)$. Notice that if $(a, b) \sim (c,d)$, then $l(a)=l(c)$ and $l(b)=l(d)$.
  
  \begin{lemma}\label{equiv}The relation $\sim$ is an equivalence.\end{lemma}
  
 \begin{proof} It is clear that $\sim$ is symmetric. Let $(a, b) \in \Sigma$, by $(C)$ for any $a\in S$ there exist $x \in S$ with $l(x)=r(a)$, so that $\sim$ is reflexive.\par
 \medskip
  Suppose that $(a, b) \sim (c,d) \sim (p, q)$. Then there are elements $x, y, \bar{x}, \bar{y}$ in $S$ with 
   \[xa = yc\ \mbox{and}\ xb = yd,\]
   \[\bar{x}c = \bar{y}p\ \mbox{and}\ \bar{x}d = \bar{y}q,\]
   where\[r(x) = r(y), l(x)= r(a), l(y) = r(c),\]
   and 
   \[r(\bar{x}) = r(\bar{y}), l(\bar{x})= r(c), l(\bar{y})= r(p).\]
 By Condition  $(C)$,  for $y, \bar{x}$ there exist  $s, t \in S$ such that  $s\bar{x}=ty$\ where
\[s \in H_{r(s),r(\bar{x})-l(\bar{x})+max\big(l(\bar{x}),l(y)\big)}  ,  t \in H_{r(s),r(y)-l(y)+max\big(l(\bar{x}),l(y)\big)}. \] Since $l(\bar{x})=r(c)=l(y)$, then $l(s)=r(\bar{x})$ and $l(t) = r(y)=r(x)$. Now,
 \[txa = tyc = s\bar{x}c = s\bar{y}p,\]
 and 
 \[txb = tyd = s\bar{x}d = s\bar{y}q.\]
 Hence  $txa = s\bar{y}p$ and $txb = s\bar{y}q$  where  $tx \in H_{r(s), r(a)}, s\bar{y} \in H_{r(s),r(p)}$. We have
 \[l(tx) = r(a), l(s\bar{y})=r(p)\ \mbox{and}\ r(tx) = r(s\bar{y}),\]
  that is,  $(a,b)\sim (p,q)$. Thus $\sim$ is transitive.\end{proof}
  
 We write the  $\sim$-equivalence class of  $(a,b)$ as $[a, b]$ and denote by  $Q$ the set of all 
   $\sim$-equivalence classes of  $\Sigma$. If  $[a, b], [c, d] \in Q$, then by  $(C)$ for  $b$ and  $c$
   there exist  $x, y$ such that  $xb = yc$  where  \[x \in H_{r(x), r(a)-l(b)+max\big(l(b),l(c)\big)}, y \in H_{r(x), r(c)-l(c)+max\big(l(b),l(c)\big)}\] and it is easy to see that 
   \[r(xa) = r(xb) = r(yc) = r(yd)=r(x)=r(y)\]
  and we deduce that $[xa, yd] \in Q$. Define a multiplication on $Q$ by 
  \[[a, b][c,d] = [xa, yd]\ \mbox{where}\ xb = yc\]
  and  $x \in H_{r(x), r(b)-l(b)+max\big(l(b),l(c)\big)}, y \in H_{r(x), r(c)-l(c)+max\big(l(b),l(c)\big)}$.
  
 \begin{lemma}\label{well} The given multiplication is well defined.\end{lemma}
 \begin{proof} Suppose that  $[a_1, b_1] = [a_2, b_2]$ and  $[c_1, d_1] = [c_2, d_2]$. Then there are 
  elements  $x_1, x_2, y_1, y_2$ in $S$ such that
  \[\begin{array}{rcl} x_1a_1=x_2a_2,\\
x_1b_1=x_2b_2,\\ 
y_1c_1=y_2c_2, \\
 y_1d_1=y_2d_2,\end{array}\]
  where
  \[ l(x_1) = r(a_1)\ , \ l(x_2) = r(a_2)\ , \ r(x_1) = r(x_2)\]
  and 
  \[ l(y_1) = r(c_1)\ , \ l(y_2) = r(c_2)\ , \ r(y_1) = r(y_2).\]
  Note that, consequently, \[l(a_1)=l(a_2),l(b_1)=l(b_2),l(c_1)=l(c_2)\ \mbox{and}\ l(d_1)=l(d_2).\] Then
  \[[a_{1},b_{1}][c_{1},d_{1}]=[xa_{1},yd_{1}] \; \mbox{where}\;  xb_{1}=yc_{1} \]  and\ $ x \in H_{r(x),r(b_{1})-l(b_{1})+max\big(l(b_{1}),l(c_{1})\big)} \ , \ y \in H_{r(x),r(c_{1})-l(c_{1})+max\big(l(b_{1}),l(c_{1})\big)}$ \\
Also,
\[[a_2,b_2][c_2,d_2]=[\bar{x}a_2,\bar{y}d_2] \ \mbox{where}\ \bar{x}b_2=\bar{y}c_2 \] and  $\bar{x} \in H_{r(\bar{x}),r(b_{2})-l(b_{2})+max\big(l(b_{2}),l(c_{2})\big)} \ , \ \bar{y} \in H_{r(\bar{x}),r(c_{2})-l(c_2)+max\big(l(b_2),l(c_2)\big)}$. \\
\\
We have to show that  $[xa_1,yd_1]=[\bar{x}a_2,\bar{y}d_2]$.\par
\bigskip
Before completing the proof of Lemma~\ref{well} we present the following lemma.
\begin{lemma}\label{nesseccary}
 Let  $a_1  ,  a_2  , b_1  , b_2 \in S$ be such that \[r(a_1)=r(b_1), r(a_2)=r(b_2)\] and suppose that $x_1 , x_2 , w_1   , w_2 \in S$ are such that \[  x_1a_1=x_2a_2 \ , x_1b_1=x_2b_2 , \ w_1a_1=w_2a_2\] where $r(x_1)=r(x_2),l(x_1)=r(a_1),l(x_2)=r(a_2)$ and $r(w_1)=r(w_2)$. Then $w_1b_{1}=w_2b_2$. \end{lemma}
 
\begin{proof} Let  $a_1  ,  a_2  , b_1  , b_2,x_1 , x_2 , w_1  , w_2$ exist as given. Note that consequently \\ $l(a_1)=l(a_2)$ and $l(b_1)=l(b_2)$. By (C) for  $w_1 , x_1$ there exist  $x,y \in S$ such that $xw_1=yx_1$ where  \[ x \in H_{r(x),r(w_1)-l(w_1)+max\big(l(w_1),l(x_1)\big)} \ , \ y \in H_{r(x),r(x_1)-l(x_1)+max\big(l(w_1),l(x_1)\big)}.\]  Then  $xw_{1}a_1 =yx_{1}a_1$, and \[xw_2a_2=xw_1a_1=yx_1a_1=yx_2a_2.\] Now, 
 \[xw_2 \in H_{r(x),l(w_2)-l(w_1)+max\big(l(w_1),l(x_1)\big)}  , yx_2 \in H_{r(x),l(x_2)-l(x_1)+max\big(l(w_1),l(x_1)\big)}\] 
and as $l(x_1)=r(a_1)$ and $l(x_2)=r(a_2)$, we have \[l(yx_2)=r(a_2)-r(a_{1})+max\big(l(w_1),r(a_{1})\big)\geqslant r(a_2)\]and \[l(xw_2)=l(w_2)-l(w_1)+max\big(l(w_1),r(a_1)\big).\] As $xw_2a_2=yx_2a_2$, then in order to use Condition (B)($i$), we have to show that  $l(xw_2)\geqslant r(a_2)$.
Since  $w_1a_1=w_2a_2$,
\begin{equation}r(w_1)-l(w_1)+max\big(l(w_1),r(a_1)\big)=r(w_1)-l(w_2)+max\big(l(w_2),r(a_2)\big) \label{eq}\end{equation} 
 so that   \[l(w_2)-l(w_1)+max\big(l(w_1),r(a_1)\big)=max\big(l(w_2),r(a_2)\big) \geqslant r(a_2)\] as desired. Therefore by condition (B)($i$),  $xw_2=yx_2$. Since  $xw_1=yx_1$ and  $x_1b_1=x_2b_2$ we have
\[xw_1b_1=yx_1b_1=yx_2b_{2}=xw_2b_2.\]
 Once we show that $r(w_1b_1),r(w_2b_2)\geqslant l(x)$, by (B)($ii$) we have $w_1b_1=w_2b_2$. Now, 
\[w_1b_1 \in H_{r(w_1)-l(w_1)+max\big(l(w_1),r(b_{1})\big),l(b_{1})-r(b_{1})+max\big(l(w_1),r(b_{1})\big)} \]
and
 \[w_2b_2 \in H_{r(w_1)-l(w_2)+max\big(l(w_2),r(b_2)\big),l(b_1)-r(b_2)+max\big(l(w_2),r(b_2)\big)},\] 
so that
\[\begin{array}{rcl} r(w_{1}b_{1})&=&r(w_1)-l(w_1)+max\big(l(w_1),r(a_1)\big)\quad \mbox{as} \ l(x_1)=r(a_1)=r(b_1)\\ &=& r(w_1)-l(w_1)+max\big(l(w_1),l(x_{1})\big) \\ &=&l(x)\end{array}\] 
and
 \[\begin{array}{rcl}r(w_2b_2)&= &r(w_1)-l(w_2)+max\big(l(w_2),r(a_2)\big)\quad \mbox{as} \ r(b_2)=r(a_2)\\ &= &r(w_1)-l(w_1)+max\big(l(w_1),r(a_{1})\big)\quad \mbox{by} ~\eqref{eq}\\ &=&r(w_1)-l(w_1)+max(l(w_1),l(x_{1}))\quad l(x_1)=r(a_1)\\ &=&l(x).\end{array}\] The proof of the Lemma is complete.\end{proof}

Returning to the proof of Lemma ~\ref{well}, by (C) for  $xa_1$ and  $\bar{x}a_2$ there exist  $w,\bar{w}$ such that  $ wxa_1=\bar{w}\bar{x}a_2$  where \[w\in H_{r(w),r(xa_1)-l(xa_1)+max\big(l(xa_1),l(\bar{x}a_2)\big)} \ \mbox{and}\  \bar{w}\in H_{r(w),r(\bar{x}a_2)-l(\bar{x}a_2)+max\big(l(xa_1),l(\bar{x}a_2)\big)}.\]
Using the fact that $l(b_1)=l(b_2),l(c_1)=l(c_2)$ and $l(a_1)=l(a_2)$, 
it is easy to see that $l(xa_1)=l(\bar{x}a_2)$. Therefore \[l(w)=r(xa_1)=r(x)\ \mbox{and} \ l(\bar{w})=r(\bar{x}a_2) =r(\bar{x}).\] Hence
 $r(wx)=r(w)=r(\bar{w})=r(\bar{w}\bar{x})$. \par
 \medskip
  Now,  $x_1a_1=x_2a_2, x_1b_1=x_2b_2$ and $wxa_1=\bar{w}\bar{x}a_2$, so that by Lemma ~\ref{nesseccary} we have  $ wxb_1=\bar{w}\bar{x}b_2$.\\
  \\
  We also have $xb_1=yc_1$ and $ \bar{x}b_2=\bar{y}c_2$, and so  $wyc_1=\bar{w}\bar{y}c_2$. Thus  \[y_1c_1=y_2c_2  , y_1d_1=y_2d_2\ \mbox{and}\ wyc_1=\bar{w}\bar{y}c_2.\] Since $r(wy)=r(\bar{w}\bar{y})$, by using the above lemma again we have $ wyd_1=\bar{w}\bar{y}d_2$. \\
 Hence $[xa_1,yd_1]=[\bar{x}a_2,\bar{y}d_2]$. This completes the proof of Lemma ~\ref{well}. \end{proof}
The next lemma is useful in verifying that the given multiplication is associative. The proof follows immediately from the fact that $l(ab)\geqslant l(b)$ , $l(de)\geqslant l(e)$, and (B)($i$).
\begin{lemma}\label{subsidiarylemma}
Let  $a , b , c ,d , e  \in S$. If  $abc=dec$ and  $l(b)\geqslant r(c), l(e)\geqslant r(c)$, then  $ab=de$.\end{lemma}

\begin{lemma}\label{assocclaw} 
The given multiplication is associative. \end{lemma}

\begin{proof} Let $[a,b],[c,d],[p,q] \in Q$. Then by using the definition of multiplication in $Q$ we have
\[ \big([a,b][c,d]\big)[p,q]=[xa,yd][p,q] \ \mbox{where}\ xb=yc\]
and  $x \in H_{r(x),r(b)-l(b)+max\big(l(b),l(c)\big)} ,  y \in H_{r(x),r(c)-l(c)+max\big(l(b),l(c)\big)}$ for some $x,y \in S$ and then
\[\big([a,b][c,d]\big)[p,q]=[wxa,\bar{w}q] \ \mbox{where} \ wyd=\bar{w}p\]
and  $w \in H_{r(w),r(yd)-l(yd)+max\big(l(yd),l(p)\big)},
\bar{w} \in H_{r(w),r(p)-l(p)+max\big(l(yd),l(p)\big)}$ for some $w,\bar{w}\in S$. Similarly,
\[ [a,b]\big([c,d][p,q]\big)=[a,b][\bar{x}c,\bar{y}q] \ \mbox{where}\ \bar{x}d=\bar{y}p\]
and  $\bar{x} \in H_{r(\bar{x}),r(d)-l(d)+max\big(l(d),l(p)\big)} ,  \bar{y} \in H_{r(\bar{x}),r(p)-l(p)+max\big(l(d),l(p)\big)}$, and then 
\[ [a,b]\big([c,d][p,q]\big)=[za,\bar{z}\bar{y}q] \ \mbox{where} \ zb=\bar{z}\bar{x}c\]
and $z \in H_{r(z),r(b)-l(b)+max\big(l(b),l(\bar{x}c)\big)}$ and
$\bar{z} \in H_{r(z),r(\bar{x}c)-l(\bar{x}c)+max\big(l(b),l(\bar{x}c)\big)}$.\\
\\
To complete our proof we have to show that  $[wxa,\bar{w}q]=[za,\bar{z}\bar{y}q]$. That is, we need to show that there are $t,h\in S$ such that $twxa=hza$ and $t\bar{w}q=h\bar{z}\bar{y}q$ with  
\[r(t)=r(h),l(t)=r(wxa)\ \mbox{and}\ l(h)=r(za).\]
By Condition (C) for  $wx,  z$ there exist  $ h , t \in S$ such that  $twx=hz$ where \[t\in H_{r(t),r(wx)-l(wx)+max\big(l(wx),l(z)\big)} , h\in H_{r(t),r(z)-l(z)+max\big(l(wx),l(z)\big)},\] and so $twxa=hza$ and $twxb=hzb$. Since  $xb=yc$ and $zb=\bar{z}\bar{x}c$ we have  $twyc=h\bar{z}\bar{x}c$. But
\[l(y)= r(c)-l(c)+max\big(l(b),l(c)\big)\geqslant r(c)\]and\[ l(\bar{x})=r(d)-l(d)+max\big(l(d),l(p)\big)= r(c)-l(d)+max\big(l(d),l(p)\big)\geqslant r(c).\] By Lemma ~\ref{subsidiarylemma} we have  $twy=h\bar{z}\bar{x}$ and so  $twyd=h\bar{z}\bar{x}d$. Now,  $wyd=\bar{w}p$ and  $\bar{x}d=\bar{y}p$, so that $t\bar{w}p=h\bar{z}\bar{y}p$. But
  \[l(\bar{w})= r(p)-l(p)+max\big(l(yd),l(p)\big) \geqslant r(p)\] and \[l(\bar{y})=r(p)-l(p)+max\big(l(d),l(p)\big)\geqslant r(p),\]so that by Lemma~\ref{subsidiarylemma}, $t\bar{w}=h\bar{z}\bar{y}$. Hence  $t\bar{w}q=h\bar{z}\bar{y}q$. It remains to prove that
 \[l(t)=r(wxa)\ \mbox{and}\ l(h)=r(za).\] 
Since \[l(t)=r(wx)-l(wx)+max\big(l(wx),l(z)\big)\] and \[l(h)=r(z)-l(z)+max\big(l(wx),l(z)\big).\] Calculating, we have 
 \begin{align} 
r(wx)&=r(w)\label{1}\\
l(wx)&=l(x)-l(yd)+max\big(l(p),l(yd)\big)\label{2}\\
l(z)&=r(b)-l(b)+max\big(l(b),l(\bar{x}c)\big)\label{3}
\end{align}
\begin{align}
\mbox{and} \ l(\bar{x}c)&=l(c)-l(d)+max\big(l(d),l(p)\big)\label{4}\\
l(yd)&=l(d)-l(c)+max\big(l(b),l(c)\big). \label{5} 
\end{align}
Since $r(wx)=r(wxa)$ and $r(z)=r(za)$, once we show that $l(z)=l(wx)$, we will have \[l(t)=r(wx)=r(wxa)\ \mbox{and} \ l(h)=r(z)=r(za).\] 
It is convenient to consider separately two cases. \\
\\
\underline{Case($i$): $l(c)\geqslant l(b)$.} We have $l(yd)=l(d)$ and $l(x)=r(b)-l(b)+l(c)$. If $l(d)\geqslant l(p)$, then from ~\eqref{4} we have $l(\bar{x}c)=l(c)$. From ~\eqref{2} and ~\eqref{3}, 
\[l(wx)=l(x)=r(b)-l(b)+l(c)=l(z).\]
If, on the other hand, $l(d)\leqslant l(p)$, then $l(\bar{x}c)=l(c)-l(d)+l(p)$. From ~\eqref{2} and ~\eqref{3}, 
\[l(wx)=l(x)-l(d)+l(p)\] 
and
\[l(z)=r(b)-l(b)+max\big(l(b),l(c)-l(d)+l(p)\big).\]
Since $l(c)\geqslant l(b)$ and $l(d)\leqslant l(p)$, then $l(c)-l(d)+l(p)\geqslant l(b)$. Thus

\[\begin{array}{rcl}l(z)&=&r(b)-l(b)+l(c)-l(d)+l(p)\\ &=&l(x)-l(d)+l(p)\\ &=&l(wx). \end{array}\]

\underline{Case($ii$): $l(c)\leqslant l(b)$.} We have $l(yd)=l(d)-l(c)+l(b)$ and $l(x)=r(b)$. If $l(d)\geqslant l(p)$, then $l(\bar{x}c)=l(c)$. From ~\eqref{2} and ~\eqref{3}, 
\[l(wx)=l(x)-l(d)+l(c)-l(b)+max\big(l(p),l(d)-l(c)+l(b)\big)\]
and
\[l(z)=r(b)-l(b)+max\big(l(b),l(c)\big)=r(b)=l(x).\]
Since $l(d)\geqslant l(p)$ and $l(c)\leqslant l(b)$ we have $l(d)-l(c)+l(b)\geqslant l(d)\geqslant l(p)$. Then $l(wx)=l(x)$. Hence $l(z)=l(x)=l(wx)$. \\
\\
If, on the other hand, $l(d)\leqslant l(p)$, then from ~\eqref{4} we have $l(\bar{x}c)=l(c)-l(d)+l(p)$. From ~\eqref{2} and ~\eqref{3}, 
\[l(wx)=l(x)-l(d)+l(c)-l(b)+max\big(l(p),l(d)-l(c)+l(b)\big)\]
and
\[l(z)=r(b)-l(b)+max\big(l(b),l(c)-l(d)+l(p)\big).\]
Once again, there are two cases. If $l(c)-l(d)+l(p)\geqslant l(b)$, then \[l(p)\geqslant l(d)-l(c)+l(b)\] and so 
\[\begin{array}{rcl}
l(wx)&=&l(x)-l(d)+l(c)-l(b)+l(p)\\ &=&r(b)-l(d)+l(c)-l(b)+l(p)\\ &=&l(z).\end{array}\]
If, on the other hand, $l(c)-l(d)+l(p)\leqslant l(b)$, then $l(p)\leqslant l(d)-l(c)+l(b)$. Hence \[l(wx)=l(x)=r(b)=l(z).\] This completes the proof of the lemma.
\end{proof}

Now we aim to show that $Q$, which we have constructed, is a semigroup of left I-quotients of $S$. First we show that $S$ is embedded in $Q$.\par
\bigskip
Let $a\in S$. Then $a \in H_{r(a),l(a)}$ and as seen earlier, there exist $x\in S$ with  $l(x)=r(a)$. Then $xa \in H_{r(x),l(a)}$ and $[x,xa]\in Q$. If $y \in S$ with $l(y)=r(a)$, then  $ya \in H_{r(y),l(a)}$ and again $[y,ya]\in Q$. By (C) there exist  $s,t \in S$   with  $sx=ty$ (and so  $sxa=tya$),  where   $ s \in H_{r(s),r(x)} , t \in H_{r(s),r(y)}$. Hence  $[x,xa]=[y,ya]$.  There is therefore a well-defined mapping  $\theta:S \longrightarrow Q$ defined by  $a\theta=[x,xa]$ where  $x\in H_{r(x),r(a)}$.  
 \begin{lemma}\label{embed} 
The semigroup $S$ is embedded in  $Q$.\end{lemma}
 \begin{proof}
Suppose that $a\theta=b\theta$, that is, $[x,xa]=[y,yb]$ where  $ x \in H_{r(x),r(a)}$ and $y \in H_{r(y),r(b)}$, then there exist  $s,t \in S$ such that  $sx=ty$ and  $ sxa=tyb$ where   $l(s)=r(x),  l(t)=r(y)$ and  $r(s)=r(t)$. We claim that  $a=b$.\par
 \bigskip
 Since  $sxa=tyb=sxb$, once we show that $r(a),r(b)\geqslant l(sx)$ we can use (B)($ii$) to get $a=b$. Now, it is easy to see that
 \[sx\in H_{r(s),r(a)}\ \mbox{and} \ ty\in H_{r(s),r(b)}\]
  and so $l(sx)=r(a)$ and $l(ty)=r(b)$. But  $sx=ty$, so that  $r(a)=r(b)=l(sx)$. Hence  $a=b$ and so $\theta$ is 1-1, our claim is established. \par
 \medskip 
  To show  that  $\theta$  is a homomorphism, let  $a\theta=[x,xa]$ and $b\theta=[y,yb]$ where $x\in H_{r(x),r(a)}$ and $y\in H_{r(y),r(b)}$. Then 
 \[ a\theta b\theta=[x,xa][y,yb]=[wx,\bar{w}yb]\; \mbox{where} \; wxa=\bar{w}y \]  and  $w \in H_{r(w),r(xa)-l(xa)+max\big(l(xa),l(y)\big)}  , \bar{w} \in H_{r(w),r(y)-l(y)+max\big(l(xa),l(y)\big)}$. Hence \[a\theta b\theta=[wx,wxab].\]
  Notice that  \[r(xa)=r(x),l(xa)=l(a)\ \mbox{and}\ l(y)=r(b)\] so that  $w \in H_{r(w),r(x)-l(a)+max\big(l(a),r(b)\big)}$. Then  \[wx\in H_{r(w),r(a)-l(a)+max\big(l(a),r(b)\big)}=H_{r(w),r(ab)}. \]
 It follows that $(ab)\theta=[wx,wxab]=a\theta b\theta$.
  \end{proof}

The main purpose of the following is to show that $Q$ is a bisimple inverse $\omega$-semigroup and $S$ is a left I-order in $Q$. First we need the following simple but useful lemma.
\begin{lemma}\label{special}
 Let $[a,b] \in Q$. Then $[a,b]=[xa,xb]$ for any  $x\in S$ with  $l(x)=r(a)$. \end{lemma}
  \begin{proof} It is clear that $r(xa)=r(x)=r(xb)$, so that $[xa,xb] \in Q$.
 By (C) for  $a$ and  $xa$    there exist  $t ,  z \in S$ such that  $ta=zxa$ where  \[t \in H_{r(t),r(a)-l(a)+max\big(l(a),l(xa)\big)} ,  z \in H_{r(t),r(xa)-l(xa)+max\big(l(a),l(xa)\big)}.\]Since  $l(xa)=l(a)$ and $r(xa)=r(x)$, we have $l(t)=r(a)$ and  $l(z)=r(xa)= r(x)$. Also,  $l(zx)=r(a)$. Hence by (B)($ii$),  $t=zx$ and so  $tb=zxb$. Thus \[[a,b]=[xa,xb].\] \end{proof}  

\begin{lemma}\label{mm}Let $[a,b],[b,c] \in Q$. Then \[[a,b][b,c]=[a,c].\]\end{lemma}
\begin{proof}
We have \[[a,b][b,c]=[xa,yc]\] where $xb=yb$ and $x,y \in H_{r(x),r(b)-l(b)+max\big(l(b),l(b)\big)}$ so that $x,y \in H_{r(x),r(b)}.$ By (B)($i$), $x=y$. Then by Lemma ~\ref{special}, $[xa,xc]=[a,c]$.
\end{proof}

\begin{lemma}\label{regular}
The semigroup $Q$ is  regular.\end{lemma}
\begin{proof} 
Let $[a,b]\in Q$. Then  $[b,a]\in Q$ and by Lemma ~\ref{mm}, \[[a,b][b,a][a,b]=[a,b].\]
\end{proof}
Let $[a,a] \in Q$, then  by Lemma~\ref{mm} we have  $[a,a][a,a]=[a,a]$, that is, $[a,a]$\ is an idempotent in \ $Q$. Hence $\{[a,a]; a \in S\}\subseteq E(Q)$. 
\begin{lemma}\label{idempotentsform}
The set of idempotents of $Q$ is given by $E(Q)=\{[a,a]; a \in S\}$. \end{lemma}  
\begin{proof}
Let $[a,b] \in E(Q)$, then  $[a,b][a,b]=[a,b]$ and so $[xa,yb]=[a,b]$ where  $xb=ya$  for some  $x \in H_{r(x),r(b)-l(b)+max\big(l(b),l(a)\big)}  , y \in H_{r(x),r(a)-l(a)+max\big(l(b),l(a)\big)}$  so that \[xa \in H_{r(x),l(a)-l(b)+max\big(l(b),l(a)\big)}  , yb \in H_{r(x),l(b)-l(a)+max\big(l(b),l(a)\big)}. \]
Since  $[xa,yb]=[a,b]$, then there exist $t,z \in S$  such that  $ txa=za$ and $tyb=zb$ where $t \in H_{r(t),r(x)}$ and  $z \in H_{r(t),r(a)}$. It follows that $l(xa)=l(a)$ and $l(yb)=l(b)$. Hence  \[tx \in H_{r(t),r(b)-l(b)+max\big(l(a),l(b)\big)}  , ty \in H_{r(t),r(a)-l(a)+max\big(l(b),l(a)\big)}, \] 
so that  \[l(tx)\geqslant r(b)=r(a), l(ty)\geqslant r(a)=r(b)\ \mbox{and} \  l(z)=r(a)=r(b).\] By (B)($i$),  $ tx=z=ty$. From (B)($ii$) as  $ r(x)=r(y)=l(t)$, we have  $x=y$, and so   $l(x)=l(y)$, that is, $r(a)-l(b)+max\big(l(b),l(a)\big)=r(a)-l(a)+max\big(l(b),l(a)\big)$. Hence  $l(a)=l(b)$ which gives  $l(x)=r(a)=r(b)$. Since  $xb=ya=xa$ by (B)($ii$)  $a=b$.\end{proof}

 \begin{lemma}\label{chain}
The set $E(Q)$\ is w-chain. \end{lemma}
\begin{proof}
Let $[a,a]  ,[b,b] \in E(Q)$, then \[[a,a][b,b]=[xa,yb] \ \mbox{where} \  xa=yb,\]  and  $x\in H_{r(x),r(a)-l(a)+max(l(a),l(b))} \ , \ y \in H_{r(x),r(b)-l(b)+max(l(a),l(b))}.$ Hence \[[a,a][b,b]=[xa,xa]=[yb,yb].\]
If  $l(a) \geqslant l(b)$, then  $x\in H_{r(x),r(a)}$  and so  $xa\in H _{r(x),l(a)}$. By Lemma ~\ref{special} we have  $[xa,xa]=[a,a]$. If  $l(b)\geqslant l(a)$, then  $y\in H_{r(x),r(b)}$  and  $yb \in H_{r(x),l(b)}$  so that $[yb,yb]=[b,b]$ by Lemma ~\ref{special}. 
\end{proof}
Notice also from Lemma ~\ref{chain} that if $l(a)=l(b)$, then $[a,a][b,b]=[a,a]=[b,b]$.\par
\medskip 
   By Lemma ~\ref{chain}, the idempotents of $Q$ form an $\omega$-chain and hence commute, by Lemma ~\ref{regular}; the following Lemma is clear.
  \begin{lemma}\label{in} The semigroup $Q$ is inverse. \end{lemma}
   \begin{lemma}\label{bi} The semigroup  $Q$ is a bisimple inverse semigroup.\end{lemma}
  \begin{proof} 
To show that  $Q$   is a bisimple inverse semigroup, we need to prove that, for any two idempotents   $[a,a] , [b,b]$  in  $E(Q)$, there is   $q$ in  $Q$ such that  $qq^{-1}=[a,a]$ and  $q^{-1}q=[b,b]$.\\ 
\\
By (A),  $S\varphi$ is a  left I-order in  $\mathcal{B}$. By Lemma ~\ref{straightlaw}, $S\varphi$ is straight, so that for  $(l(a),l(b))$ there exist  $c , d$ in  $S$ such that  \[(l(a),l(b))=c\varphi^{-1}d\varphi\; \mbox{where} \; c\varphi\,\mathcal{R}\,d\varphi  \mbox{ in} \ \mathcal{B}, \] so that  $c\varphi=(u,l(a))$ and  $d\varphi=(u,l(b))$ for some  $u\in \mathbb{N}^0$. Hence   $q=[c,d]\in Q$. By Lemma ~\ref{mm},  $qq^{-1}=[c,d][d,c]=[c,c]$ and, similarly, $q^{-1}q= [d,d]$. By the argument following Lemma ~\ref{chain}, $[c,c]=[a,a]$ and $[d,d]=[b,b]$, as required.  
\end{proof}

  The following lemma throws full light on the relationship between $S$ and $Q$.
 \begin{lemma}\label{liorder}
Every element of $Q$ can be written as $(a\theta)^{-1}b\theta$, where $a,b \in S$.\end{lemma}

\begin{proof}
Suppose that  $q=[a,b] \in Q$. In view of Lemma~\ref{embed}  $a\theta=[x,xa]$ and $b\theta=[y,yb]$ respectively, for some  $x\in H_{r(x),r(a)}$ and\ $y\in H_{r(y),r(b)}$. Hence 
\[\begin{array}{rcl} (a\theta)^{-1}b\theta&=&[xa,x][y,yb]\\
&=&[txa,hyb] \quad \mbox{where}\ tx=hy, r(t)=r(h), l(t)=r(x) \ \mbox{and}\ l(h)=r(y) \\ &=&[txa,txb] \quad \mbox{where} \ l(tx)=r(a)\\ &=&[a,b] \quad \mbox{by Lemma ~\ref{special}}.\end{array}\] 
 \end{proof}

 From Lemmas  ~\ref{embed},  ~\ref{chain}, ~\ref{in}, ~\ref{bi} and ~\ref{liorder} we deduce that  $S$ is a straight left I-order in a bisimple inverse  $\omega$-semigroup.\end{proof}


\begin{thebibliography}{111}
\bibitem{clifford} A. H. Clifford and G. B. Preston, The algebraic theory of semigroups, Vol. 1,
   Mathematical Surveys 7, \emph{American Math. Soc.} (1961).
 
  %\bibitem{fountain} J. B. Fountain and M. Petrich, Brandt semigroups of quotients, 
  \emph{Math. Camb. Phil. Soc.}, \textbf{98}(1985), 413-426. 
  \bibitem{pjhon} J. B. Fountain and Mario Petrich, Completely 0-simple semigroups of quotients, \emph{Journal of Algebra} \textbf{101}, 365-402(1986).
  
  
  %\bibitem{NG} N. Ghroda, Primitive inverse semigroups of left I-quotients, ?????????????? 
  \bibitem{NGb} N. Ghroda, Bicyclic semigroups of left I-quotients, in preparation.
  \bibitem{GG} N. Ghroda and V. Gould, Inverse semigroups of I-quotients.http://www-users.york.ac.uk/~varg1/gpubs.htm
  \bibitem{onordersGould} V. Gould, Orders in semigroups, in \emph{Contributions to General Algebra 5}, 163-169,
   Verlag Hölder-Pichler-Tempsky, Vienna, 1987.
  
  \bibitem{bisGould} V. Gould, Bisimple inverse \ $\omega$-semigroup of left quotients, \emph{Pro. London Math. Soc.} \textbf{52}(1986), 95-118.
   \bibitem{howie} J. M. Howie, Fundamentals of semigroup theory, Oxford University Press 1995.
   \bibitem{Reilly} N. R. Reilly, Bisimple $\omega$-semigroup, \emph{Proc. Glasgow Math. Assoc.} part 3, \textbf{7}(1966), 160-67.
\end{thebibliography}
\end{document}